\documentclass[11pt,english]{article}

\usepackage[spanish, english]{babel}
\usepackage{amsmath,amssymb,amsthm, mathrsfs, color, mathtools, pgf,tikz, titlesec, graphicx}
\usepackage{enumitem}
\usetikzlibrary{arrows}
\usepackage{relsize}

 \usepackage[colorlinks=true]{hyperref}

\hypersetup{urlcolor=blue, citecolor=red}

\newtheorem{Teorema}{Teorema}[section]

\newtheorem{Proposition}[Teorema]{Proposition}
\newtheorem{Lemma}[Teorema]{Lemma}
\newtheorem{Theorem}[Teorema]{Theorem}

\newtheorem{Corollary}[Teorema]{Corollary}

\newtheorem{Definition}[Teorema]{Definition}
\newtheorem{Remark}[Teorema]{Remark}

\begin{document}


 
\selectlanguage{english}

\begin{center}

{\LARGE \textbf{Horocyclic and geodesic orbits on geometrically infinite surfaces with variable negative curvature}}\\
  \vspace{.50in}  Victoria García

\end{center}

\begin{abstract}
Here we study the behaviour of the horocyclic orbit of a vector on the unit tangent bundle of a geometrically infinite surface with variable negative curvature, when the corresponding geodesic ray is almost minimizing and the injectivity radius is finite. 
\end{abstract}

\section{Introduction}

Let $M$ be an orientable geometrically infinite surface with a complete Riemanninan metric of negative curvature, and let $\tilde{M}$ be its universal cover. Let us suppose that $\Gamma$ is the fundamental group of $M$. We can see it as a subgroup of the group of orientation preserving isometries of $\tilde{M}$, that is $\Gamma<Isom^+(\tilde{M})$. We can write then $$M= \Gamma\textit{\textbackslash}\tilde{M}.$$ If $T^1M$ and $T^1\tilde{M}$ are the unit tangent bundles of $M$ and $\tilde{M}$ respectively, then we can also write: $$T^1M=\Gamma \textit{\textbackslash}T^1\tilde{M}.$$

If the curvature of $\tilde{M}$ has an upper bound $-\kappa ^2$, with $\kappa >0$, then the geodesic flow on $T^1\tilde{M}$, denoted by $g_\mathbb{R}$, happens to be an Anosov flow (see appendix of \cite{WB}), and this flow descends to $T^1M$. The strong stable manifold for the geodesic flow defines a foliation, 
(as we shall see in section \ref{prel}), which is the stable horocycle flow, denoted by $h_\mathbb{R}$, and it also descends to $T^1M$. 

A work of Hedlund shows that if the surface $M$ has constant negative curvature and it is compact, then the horocycle flow is minimal on the unitary tangent bundle. This means that the only closed non empty invariant set for the horocycle flow is $T^1M$ (\cite{GH}). Later, F. Dal'Bo generalizes this result to compact surfaces of variable negative curvature (\cite{FD2}). She also proves that in case the fundamental group is finitely generated, then all the horocycles on the non-wandering set are either dense or closed, motivating the interest for studying geometrically infinite surfaces.

S. Matsumoto studied a family of geometrically infinite surfaces of constant curvature, for which he proves that the hororycle flow on the unit tangent bundle does not admit minimal sets (\cite{SM}). Also Alcalde, Dal'Bo, Martinez and Verjovsky (\cite{ADMV}) studied this family of surfaces which appear as leaves of foliations, and proved the same result in this context. More recently, A. Bellis studied the links between geodesic and horcycle orbits for some geometrically infinite surface  of constant negative curvature (\cite{AB}). In particular, his result implies Matsumoto's result. 

Our aim was to determine whether or not the results of these last works are still valid if we don't have the hypothesis of constant curvature, and to provide arguments that do no depend on specific computations which are valid on constant negative curvature. 

Let us introduce the following definitions.

\begin{Definition}\em \label{injra}
Let $p$ be a point on a surface $M$. The \emph{injectivity radius} of $p$ is defined as $$ \mathrm{Inj}(p):=\min_{Id\neq \gamma\in \Gamma} d(\tilde{p},\gamma(\tilde{p})),$$ where $\tilde{p}$ is any lift of $p$ to the universal cover $\tilde{M}$ of $M$. 
\end{Definition}

\begin{Definition}\em
If $v\in T^1M$, the \emph{geodesic ray} $v[0,\infty)$ is the projection of the future geodesic orbit $\{g_t(v):t\in [0,\infty)\}$ of $v$ on $M$. Also $v(t)$ will be the projection on $M$ of $g_t(v)\in T^1M$. 
\end{Definition}

\begin{Definition}\em
A geodesic ray $v[0,\infty)$ on $M$ is said to be \emph{almost minimizing} if  there is a positive real number $c$ such that $$d(v(t),v(0))\geq t-c\ \forall t\geq 0.$$ 
\end{Definition}

\begin{Definition}\em \label{injrad}
Let $v[0, \infty)$ be a geodesic ray, then we define its injectivity radius as $$\underline{\mathrm{Inj}}(v[0, \infty)):= \liminf_{t\longrightarrow \infty} \mathrm{Inj}(v(t)).$$
\end{Definition}

We prove the following theorem: 

\begin{Theorem}
\label{teorema}
Let M be an orientable geometrically infinite surface with a complete Riemanninan metric of negative curvature. Let $v\in T^1M$ such that $v[0,\infty)$ is an almost minimizing geodesic ray with finite injectivity radius $a$, and such that $h_\mathbb{R}(v)$ is not closed. Then there is a sequence of times $\tau_n$ going to $\infty$ such that $g_{\tau_n}(v)\in \overline{h_\mathbb{R}(v)}$ for all $n$. And even more, the set  $\mathcal{I}=\{t\in \mathbb{R}:g_t(v)\notin \overline{h_\mathbb{R}(v)}\}$ only contains intervals of bounded length. 
\end{Theorem}

This is the result that was proved by Bellis (\cite{AB}) for surfaces of constant negative curvature. Our proof takes some ideas of Bellis's proof, but introduces a different approach in some parts.

The following result generalizes the one proved by Matsumoto in the context of constant negative curvature. 

\begin{Definition}\em \label{tight}
Let $M$ be a noncompact Riemannian surface with variable negative curvature, and let $\Gamma$ be its fundamental group. We say that $M$ is \textit{tight} if $\Gamma$ is purely hyperbolic, and $M$ can be written as $$M:=M_1\cup M_2\cup...$$ where $M_n\subset M_{n+1}\ \forall \ n$, and each $M_n$  is a compact, not necessarily connected submanifold of $M$ with boundary $\partial M_n$, and the boundary components are closed geodesics whose lengths are bounded by some uniform constant $A\in \mathbb{R}$. 
\end{Definition}

\begin{Corollary}\label{theorema2}
If $M$ is a tight surface, there are no minimal sets for the horocycle flow on $T^1M$. 
\end{Corollary}

I would like to thank my advisors M. Martínez and R. Potrie for their invaluable help while writing this article. Time spent in conversations with F. Dal'Bo was very helpful for understanding horocycle flows, specially in geometrically infinite and variable curvature contexts. I would also like to thank S. Burniol for a very useful hint to prove Propositions \ref{prop1} and \ref{prop2}, and for a careful reading of this text.

\section{Preliminaries}\label{prel}

If $u\in T^1\tilde{M}$, we denote by $g_\mathbb{R}(u)$ the geodesic passing through $u$ in $T^1\tilde{M}$, and by $u(\mathbb{R})$ the projection of this geodesic on $\tilde{M}$. So that $u(t)$ will denote the projection of $g_t(u)$ on $\tilde{M}$. We denote by $h_\mathbb{R}(u)$ the horocycle passing through $u$.

We will also denote by $\pi$ the projection from $T^1M$ to $M$, and by $\hat{\pi}$ the projection from $T^1\tilde{M}$ to $T^1M$. 

\subsection{Boundary at Infinity}

The boundary at infinity is the set of endpoints of all the geodesic rays. Here we give a formal definition.

\begin{Definition}\em
\label{obslim}
We are going to say that the geodesics directed by two vectors $\tilde{v}$ and $\tilde{v}' \in T^1\tilde{M}$ have the same endpoint if $$\sup_{t>0}d(\hat{\pi}(g_t(\tilde{v})),\hat{\pi}(g_t(\tilde{v}')))<\infty.$$ When this holds, we write: $\tilde{v}\sim_* \tilde{v}'$.
\end{Definition}

\begin{Definition}\em
The \em boundary at infinity \em of $\tilde{M}$ will be the set defined by $$\partial_{\infty}\tilde{M}:=T^1\tilde{M}\textit/ \sim_*.$$ 
\end{Definition}

For any $\tilde{v}\in T^1\tilde{M}$, we denote by $\tilde{v}(\infty)$ its equivalence class by the relation $\sim_*$.

Given two different points $\xi$ and $\eta$ of $\partial_\infty\tilde{M}$, we will denote by $(\xi,\eta)$ the geodesic on $\tilde{M}$ joining them. 

The action of $\mathrm{Isom}^+(\tilde{M})$ on $\tilde{M}$ can be naturally extended to $\tilde{M}\cup \partial_\infty\tilde{M}$.

An isometry of $\mathrm{Isom}^+(\tilde{M})$ is said to be 
\begin{itemize}
\item \emph{hyperbolic} if it has exactly two fixed points and both of them lie on $\partial_\infty\tilde{M}$. 
\item \emph{parabolic} if it has a unique fixed point and it lies on $\partial_\infty\tilde{M}$.
\item \emph{elliptic} if it has at least one fixed point in $\tilde{M}$. 
\end{itemize}

Every isometry $\mathrm{Isom}^+(\tilde{M})$ is either hyperbolic, elliptic or parabolic  (see \cite{BKS}). 

\subsection{The Busemann function and horocycles}

The Busemann function is one of the main tools we need to describe horocycles and their properties. In this section we show how to construct this function. 

\begin{Definition} For $x,y,z\in \tilde{M}$ we can define $$b:\tilde{M}\times \tilde{M}\times \tilde{M}\longrightarrow \mathbb{R}$$ $$b(x,y,z)=d(x,y)-d(x,z), x,y,z \in \tilde{M}.$$
\end{Definition}

As $d$ is a continuous function, $b$ is also continuous.

\begin{Remark}\label{busemannprops}
 Some properties of $b$ that can be deduced from the definition are:
\begin{enumerate}
\item $b(x,y,y)=0$ for all $x, y \in \tilde{M}$
\item $\mid b(x,y,z)-b(x,y,z')\mid \leq d(z,z')$ for all $x,y,z,z' \in \tilde{M}$
\item $b(x,y,z)=b(x,y,z')+b(x,z',z)$ for all $x,y,z,z'\in \tilde{M}$
\end{enumerate}

See chapter II.1 of \cite{WB} for proof. 

\begin{Definition}\em
Let $\{x_n\}_{n\in \mathbb{N}}$ be a sequence in $\tilde{M}$. Given $\xi \in \partial_\infty \tilde{M}$, consider a fixed point $o\in M$ and the geodesic ray $u[0,\infty)$ with $u(0)=o$ and $u(\infty)=\xi$. Consider also the geodesic rays $u_n[0,\infty)$ such that $u_n(0)=o$ and $u_n(t_{j_n})=x_n$ for some $t_{j_n}\in [0,\infty)$, and define $\xi_n:=u_n(\infty)$. Then, we say that $x_n\longrightarrow \xi$ if $\xi_n \longrightarrow \xi$ in $\partial_\infty \tilde{M}$. 
\end{Definition}

\end{Remark}

 Given $x,y \in \tilde{M}$, the map $b_y(x):\tilde{M}\longrightarrow \mathbb{R}$, defined as $$b_y(x)(z):=b(x,y,z),\ z\in \tilde{M}$$ is a continuous function on $\tilde{M}$.  Let us consider $C(\tilde{M})$ the space of continuous function on $\tilde{M}$ with the topology of the uniform convergence on bounded sets. Let $\{x_n\}$ be a sequence on $\tilde{M}$ such that $x_n\longrightarrow \xi\in \partial_\infty \tilde{M}$. Then $b_y(x_n)$ converges on $C(\tilde{M})$ to some function $B_\xi(y,\ )$. This will be the \textit{Busemann function} at $\xi$, based at $y$. Explicitly we have:
 \begin{equation}
 \label{busemann}
  B_\xi(y,z):=\lim_{x_n\longrightarrow \xi}d(x_n,y)-d(x_n,z),
 \end{equation}
where $x_n$ is a sequence on $\tilde{M}$ that goes to $\xi$. 
 
 This definition is independent of the choice of the sequence $x_n$ (see chapter II.1 of \cite{WB}).
 
 For all $\xi \in \partial_\infty\tilde{M}$, we have $B_\xi^y:\tilde{M}\longrightarrow \mathbb{R}$, where $B_\xi^y(z)=B_\xi(y,z)$ for all $z\in \tilde{M}$,  is a continuous function, and then the level set $(B_\xi^y)^{-1}(t)$ is a regular curve (meaning that it admits a $C^1$ arclength parametrization)  for all $t\in \mathbb{R}$ (see chapter IV.3 of \cite{WB}). Let us denote by $H_y(\xi,t)$ the level set $(B_\xi^y)^{-1}(t)$, and for each $p\in H_y(\xi,t)$ consider the only vector $\tilde{v}_\xi^p \in T_p^1\tilde{M}$ such that $\tilde{v}_\xi^p(\infty)=\xi$. Then the set $$\hat{H}_y(\xi,t):=\{ \tilde{v}_\xi^p:p\in H_y(\xi,t) \}$$ is the horocycle through $\tilde{v}_\xi^p$ in $T^1\tilde{M}$, for all $p\in H_y(\xi,t)$. We have that $\pi(\hat{H}_y(\xi,t))=H_y(\xi,t)$, and $\hat{H}_y(\xi,t)\subset T^1\tilde{M}$ is  the strong stable set of $\tilde{v}_\xi^p$ for the geodesic flow, which can also be parametrized by arclength. Now, the horocycle flow $h_s(v)$, pushes a vector $v$ along its strong stable manifold, through an arc of length $s$.

Given two elements $u$ and $v$ in $T^1\tilde{M}$, let $z_u=u(0)$ and $z_v=v(0)$ be their respective base points in $\tilde{M}$. Let us suppose that there are $t,s\in\mathbb{R}$ such that $g_t(u)=h_s(v)$ or, in other words, there is $t \in \mathbb{R}$ such that $g_t(h_\mathbb{R}(u))=h_\mathbb{R}(v)$. In this case $u(\infty)=v(\infty)$. Let us suppose that $u(\infty)=v(\infty)=\xi\in \partial_\infty\tilde{M}$. Then the Busemann function centered in $\xi$ evaluated at $(z_u,z_v)$ happens to be the real number $t$ mentioned above. We denote it by $$B_\xi(z_u,z_v)=t.$$



\begin{Remark}\label{rem1}
If $u\in T^1\tilde{M}$ is such that $u(0)=o$ and $u(\infty)=\xi$, then $$B_\xi(o,z)=\lim_{t\rightarrow \infty}[d(o,u(t))-d(z,u(t))]=\lim_{t\rightarrow \infty} t-d(z,u(t)).$$
\end{Remark}

\subsection{Limit set and classification of limit points}

The limit set is a special subset of the boundary at infinity. We classify limits points, and show their links with the behaviour of horocyclic orbits. 

\begin{Definition}\em
The \emph{limit set} $L(\Gamma)$ of the group $\Gamma$, is the set of accumulation points of an orbit $\Gamma z$, for some $z\in \tilde{M}$. This is well defined because all orbits have the same accumulation points (see chapter 1.4 of \cite{BKS}).  
\end{Definition}

One has $L(\Gamma)\subset \partial_\infty\tilde{M}$. Otherwise, a sub-sequence of the orbit would remain on a compact region of $\tilde{M}$, contradicting the fact that $\Gamma$ acts discontinuously on $\tilde{M}$.

In the limit set, we can distinguish two different kinds of points: 

\begin{Definition}\em
A limit point $\xi\in \partial_\infty\tilde{M}$ is said to be \emph{horocyclic} if given any $z\in \tilde{M}$, and $t\in \mathbb{R}$, there is $\gamma\in \Gamma$ such that $B_\xi(o,\gamma(z))>t$. Otherwise, we say $\xi$ is a \emph{nonhorocyclic} limit point. (See figure 1)
\end{Definition} 

\begin{figure}[ht!]  
\label{plh}
\begin{center}   
\includegraphics[width=120mm]{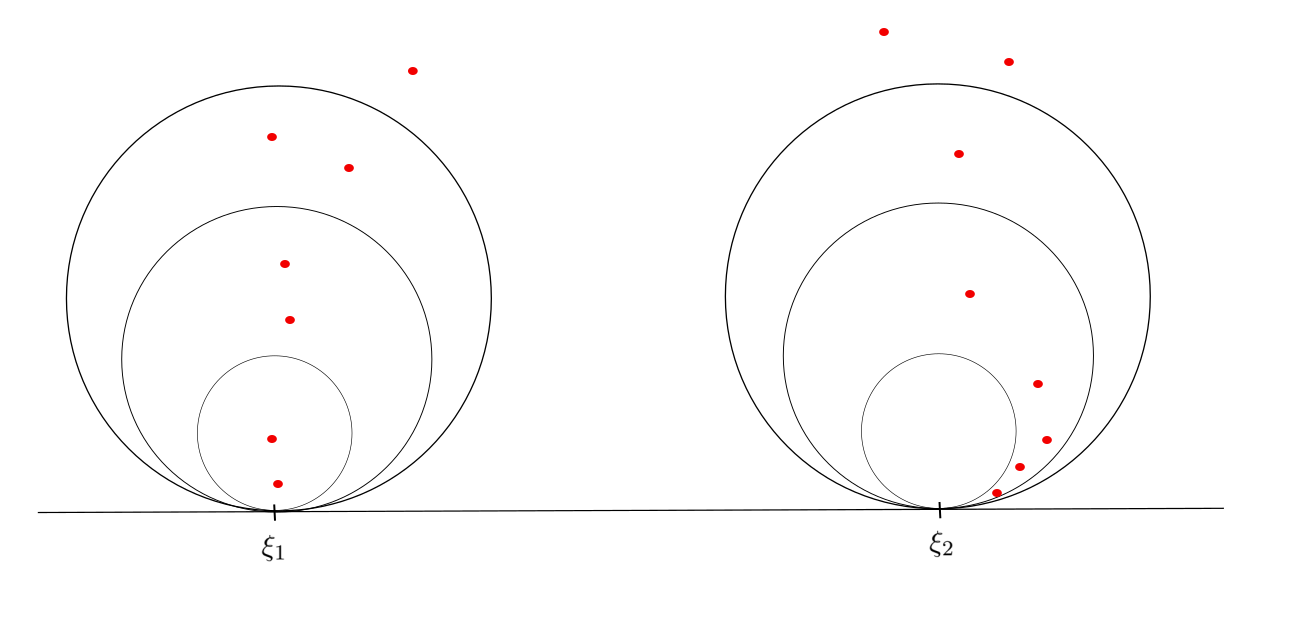}
 \end{center} 
 \caption{Here, $\xi_1$ is an horocyclic limit point. The points, representing some elements of the $\Gamma$-orbit of a point, reach all the horodisks based at $\xi_1$. The point $\xi_2$ on the other hand would be a nonhorocyclic limit point, as no element of the orbit reaches the smaller horodisk.}
 \end{figure}

The sets of the form $$\{z\in\tilde{M}:B_\xi(o,\gamma(z))>t,\},\ t\in \mathbb{R}$$ are called \emph{horodisks} based on $\xi$. So, in other words, a limit point $\xi$ is horocyclic if each horodisk based on $\xi$ intersects de orbit $\Gamma z$, for all $z\in \tilde{M}$. 

\begin{Remark}\label{remhoropoint}
Given a point $\xi\in \partial_\infty \tilde{M}$, if there is a sequence $\{\gamma_n\}\subset \Gamma$ such that $B_\xi(o,\gamma_n^{-1}(o))\xrightarrow[n\rightarrow{}\infty]\, \infty$ for any $o\in \tilde{M}$, then $\xi$ is an horocyclic limit point. 

This is because if $B_\xi(o,\gamma_n^{-1}(o))\xrightarrow[n\rightarrow{}\infty]\, \infty$, any horodisk $\{z\in\tilde{M}:B_\xi(o,\gamma(z))>t,\}$ will contain an element of the $\Gamma$-orbit of $o$.
\end{Remark}

We denote by $\Lambda_\Gamma$ the image by $\hat{\pi}$ of the set $\{\tilde{v}\in T^1\tilde{M}: \tilde{v}(\infty)\in L(\Gamma)\}$. 

\begin{Proposition}\label{propdense}
If $\xi\in L(\Gamma)$ is an horocyclic limit point, then for all $\tilde{v}\in T^1\tilde{M}$ such that $\tilde{v}(\infty)=\xi$, we have $\hat{\pi}(h_\mathbb{R}(\tilde{v}))$ is dense in $\Lambda_\Gamma$.
\end{Proposition}

The proof of this proposition can be found in \cite{FD2}, as Proposition B.

Given $v\in T^1M$ (or $T^1\tilde{M}$), we denote by $v[0,\infty)$ the geodesic ray  $g_t(v)_{t\in [0,\infty)}.$

\subsection{Almost minimizing geodesic rays}
The following proposition relates the behaviour of a geodesic ray with its endpoint. 

First, we recall the following definition:

\begin{Definition}\em
A geodesic ray $v[0,\infty)$ on $M$ is said to be \emph{almost minimizing} if  there is a positive real number $c$ such that $$d(v(t),v(0))\geq t-c\ \forall t\geq 0.$$ (See figure 2)
\end{Definition}

\begin{figure}[ht!]  
\label{cmin}
\begin{center}   
\includegraphics[width=120mm]{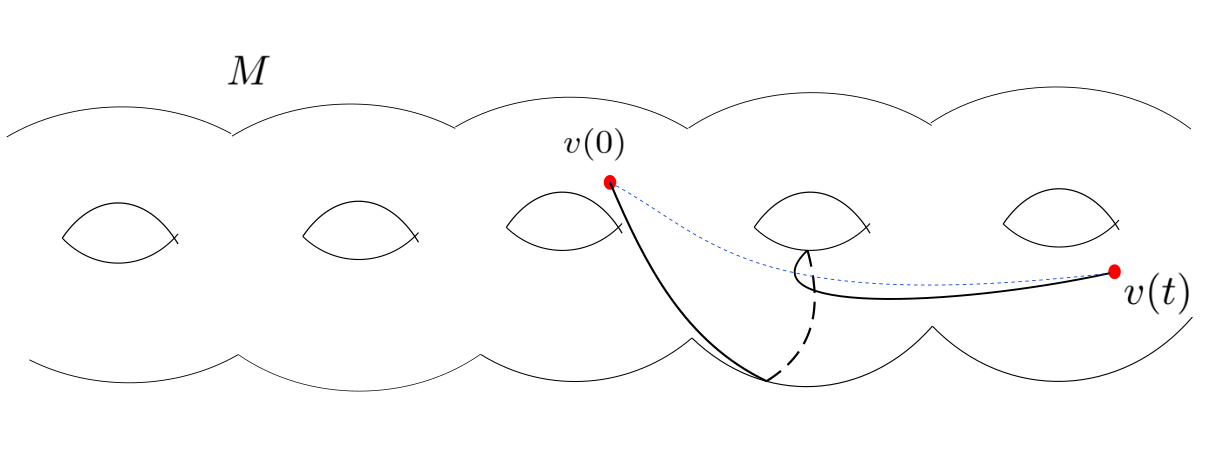}
 \end{center} 
 \caption{The projected geodesic ray starts at $v(0)$ and at time $t$ passes through the point $v(t)$. The distance between this two points is less than $t-c$. The blue dotted line represents the minimizing geodesic joining $v(0)$ and $v(t)$, which obviously has length $t$.}
 \end{figure}

\begin{Proposition}\label{propam}
Let $\xi\in L(\Gamma)$ and $\tilde{v}\in T^1\tilde{M}$ such that $\tilde{v}(\infty)=\xi$. Then the projected geodesic ray $v[0,\infty)$ over $M$, is almost minimizing if an only if $\xi$ is a nonhorocyclic limit point.
\end{Proposition}

\begin{proof}
Take a reference point $o$ and suppose without loss of generality that $\tilde{v}\in T^1_o\tilde{M}$. Let us suppose first that $\xi=\tilde{v}(\infty)$ is a nonhoroyclic limit point. Then there is a horodisk $H$ based at $\xi$ that does not contain any point of the $\Gamma$-orbit of $o$. Let us take $H=\{z\in  \tilde{M}:B_\xi(o,z)> k$, with $k>0\}$. Then for all $\gamma \in \Gamma$, one has $B_\xi (o,\gamma (o))\leq k$. And because of remark \ref{rem1}, this means that $$\lim_{t\rightarrow \infty}[d(o,\tilde{v}(t))-d(\gamma(o),\tilde{v}(t))]=\lim_{t\rightarrow \infty}[t-d(\gamma(o),\tilde{v}(t))]\leq k.$$ As this limit is not decreasing, we have for any $t>0$: $$t-d(\gamma(o),\tilde{v}(t))\leq k,$$ and then $$d(\gamma(o),\tilde{v}(t))\geq t-k.$$ As this happens for any $\gamma \in \Gamma$, down on the surface $M$, this implies that $$d(v(0),v(t))\geq t-k,$$ which means that $v[0,\infty)$ is an almost minimizing geodesic ray.
 All this implications can be reverted to prove that an almost minimizing geodesic ray, is projected from a geodesic ending on a nonhorocyclic limit point. 
\end{proof}

Also a proof of the proposition above can be found in \cite{SM}.

\subsection{Links between geodesic and horocyclic orbits}

In this section, $M$ will be an orientable geometrically infinite surface with a complete Riemanninan metric of negative curvature, $\tilde{M}$ its universal cover and $\Gamma$ its fundamental group. 

As we mentioned before, the geodesic flow on $T^1M$ is an Anosov flow (see \cite{KH}), and the stable manifolds, which are contracted by this flow (see ch. IV of \cite{WB}), have the level sets of the Busemann functions as their projections to $M$. 

The strong stable manifold of the geodesic flow is defined as follows:

\begin{Definition}\em
Consider the geodesic flow $g_t:T^1M\longrightarrow T^1M$, and take $v\in T^1M$. Then the \emph{strong stable manifold} of $v$ will be the set $$W^s(v):=\{u\in T^1M:\ d(g_t(v),g_t(u))\xrightarrow[t\rightarrow{}\infty]\, 0\}.$$
\end{Definition}

The following proposition gives a criteria we are going to use in the proof of Theorem \ref{teorema}.

\begin{Proposition}\label{Prop}
For $u \in T^1M$, if there are sequences $u_n \in T^1M$ and $r_n\in \mathbb{R}$ such that $u_n\longrightarrow u$ in $T^1M$, $r_n\longrightarrow r_0\in \mathbb{R}$, and $d(g_{t+r_n}(u_n), g_t(u))\xrightarrow[t\rightarrow{}\infty]\, 0$, then $g_{r_0}(u)\in \overline{h_{\mathbb{R}}(u)}$. 
\end{Proposition}

\begin{proof}
As $d(g_{t+r_n}(u_n), g_t(u))\xrightarrow[t\rightarrow{}\infty]\, 0$, this means that $g_{r_n}(u_n)$ is in the strong stable manifold of $u$ for all $n$. Then, there is a sequence $\{s_n\}_{n\in \mathbb{N}}$ such that $g_{r_n}(u_n)=h_{s_n}(u)$. Then $u_n=g_{-r_n}h_{s_n}(u)$. 

By hypothesis we have $u_n\xrightarrow[n\rightarrow{}\infty]\, u$, and then $g_{-r_n}h_n(u)\xrightarrow[n\rightarrow{}\infty]\, u$. It follows that $h_{s_n}(u)\xrightarrow[n\rightarrow{}\infty]\, g_{r_n}(u)$ in $T^1M$. This implies that $d(h_{s_n}(u),g_{r_n}(u))\xrightarrow[n\rightarrow{}\infty]\, 0$, and as $r_n\longrightarrow r_0$, the family $\{g_{r_n}\}$ is equicontinuous, and we finally have $$h_{s_n}(u)\xrightarrow[n\rightarrow{}\infty]\, g_{r_0}(u),$$ where $h_{s_n}(u)$ is a sequence on $h_\mathbb{R}(u)$, and so $g_{r_0}(u)\in \overline{h_\mathbb{R}(u)}$ as we wanted to see.  
\end{proof}

\section{Geometric properties of horocycles}

In this section, we prove some geometric properties of horocycles and the Busemann function, which are going to be useful tools in the proof of Theorem \ref{teorema}.

First, we introduce some additional notation:
If $\gamma\in \Gamma$ is an hyperbolic isometry $(\gamma^-,\gamma^+)$ is the axis of $\gamma$, where $\gamma^-, \gamma^+ \in  \partial_\infty \tilde{M}$ are its fixed points, we will denote by $\mathscr{C}_{\gamma}(p)$ the curve passing through $p$ whose points are at a constant distance from $(\gamma^-,\gamma^+)$. In general, if $c:\mathbb{R}\longrightarrow \tilde{M}$ is a geodesic, then $\mathscr{C}_c(p)$ will be the curve passing through $p$ whose points are at a constant distance from $c(\mathbb{R})$.

For any regular connected curve $\mathscr{C}$, and $p$, $q\in  \mathscr{C}$, we will write $[p,q]_\mathscr{C}$ to denote the arc contained in $\mathscr{C}$ joining $p$ and $q$.

Finally, for $\xi\in \partial_{\infty} \tilde{M}$ and $p\in \tilde{M}$, the horocycle based at $\xi$ passing through $p$ will be denoted by $\mathscr{H}_\xi (p)$.

\begin{Remark} \label{regular1}
For any hyperbolic isometry $\gamma\in \Gamma$ and any $p\in \tilde{M}$, the curve $\mathscr{C}_{\gamma}(p)$ is a regular curve. 
In fact, one can construct charts of $\mathscr{C}_{\gamma}(p)$ from the charts of the axis of $\gamma$, which is a geodesic.
\end{Remark}

The tools for proving the following proposition can be found in Ch. 9 of \cite{dC}.

\begin{Proposition}\label{regularorthogonal1}
The distance from a point $p\in \tilde{M}$ and any closed submanifold $\tilde{N}$ of $\tilde{M}$ is attained by a minimizing geodesic which is orthogonal to $\tilde{N}$. 
\end{Proposition}

\begin{Corollary}\label{regularorthogonal2}
The distance between two  closed submanifolds $\tilde{N}$ and $\tilde{W}$ of $\tilde{M}$ is attained by a geodesic orthogonal to both $\tilde{N}$ and $\tilde{W}$. 
\end{Corollary}

 \begin{Proposition}\label{propi}
 Consider $\zeta \in \partial_\infty \tilde{M}$, a point $p\in \tilde{M}$, and $c:\mathbb{R}\longrightarrow \tilde{M}$ a geodesic. Then: 
 
 \begin{enumerate}\label{prop1}
 \item $\mathscr{C}_c(p) \cap \mathscr{H}_\zeta (p)$ contains at most two points.
 \item Up to changing the orientation of $\gamma$, there is an increasing function $\delta:\mathbb{R}_+\longrightarrow \mathbb{R}_+$ such that if $q\in [p,\gamma(p)]_{\mathscr{C}_{\gamma}(p)}$ and $d(p,q)>\epsilon$, then $|B_\zeta(q,p)|>\delta(\epsilon)$. In particular, $|B_\zeta(p,\gamma(p))|>\delta(\epsilon)$. 
 
\end{enumerate}   
 \end{Proposition}

 \begin{proof}
 First, let us give a parametrization $a(t)$ of the curve $\mathscr{C}_{c}(p)$ such that for each $t\in \mathbb{R}$ we have that $d(a(t),c(t))$ is constant and equal to $l=d(p,c(\mathbb{R}))$, and $p=a(0)$. Let $b_t:[0,l]\longrightarrow \tilde{M}$ the geodesic joining $a(t)$ and $c(t)$ (see, figure below). Now we write $\dot{b}_t(s)=\frac{\partial b_t}{\partial s}(s)$, with $s\in [0,l]$. Also $\dot{a}$ and $\dot{c}$ will refer to $\frac{\partial a}{\partial t}$ and $\frac{\partial c}{\partial t}$ respectively. We have then that $\langle \dot{b}_t(0),\dot{a}(t)\rangle= \langle \dot{b}_t(l),\dot{c}(t)\rangle=0$. This holds since the curves parametrized by $a(t)$ and $c(t)$ are closed submanifolds of $\tilde{M}$, and in view of Corollary \ref{regularorthogonal2}, the distance between them is assumed by a geodesic perpendicular to both $c(\mathbb{R})$ and $a(\mathbb{R})$. 
 
 If we look at the Busemann function $B_\zeta^p(a(t))$ along the curve $a(t)$, where $B_\zeta^p(z):=B_\zeta(p,z)$ for all $z\in \tilde{M}$, we see that its derivative vanishes if and only if $\langle \nabla B_\zeta^p(a(t)), \dot{a}(t)\rangle=0$, as the directional derivative of a function is zero if and only if the gradient of the function is orthogonal to the direction of the derivative. We are going to show that this derivative vanishes at most for one value of $t$. If we show this, then $\mathscr{C}_c(p)$ can only meet a level set of $B_\zeta^p$ at most two times, as we want to prove.

 \begin{figure}[ht!]  
 \label{horo1}
\begin{center}   
\includegraphics[width=120mm]{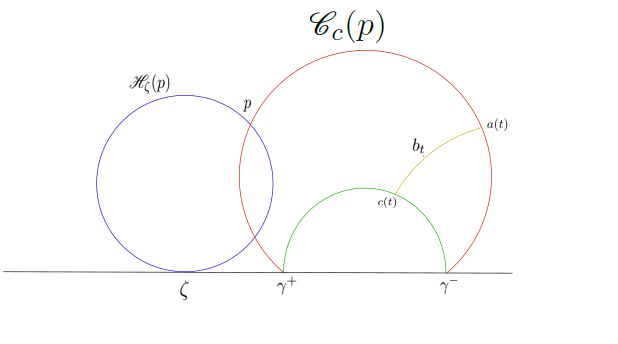}
 \end{center} 

 \end{figure}

 Suppose then that there is $t_1$ such that $\langle \nabla B_\zeta^p(a(t_1)), \dot{a}(t_1)\rangle=0$. Then, $\nabla B_\zeta^p(a(t_1))=K\dot{b}_{t_1}(0)$ for some $K\in \mathbb{R}$, as $\dot{b}_{t_1}(0) \perp \dot{a}(t_1)$. Then, the geodesic ray directed by $\dot{b}_{t_1}(0)$ (or $-\dot{b}_{t_1}(0)$) has the same endpoint as $\nabla B_\zeta^p(a(t_1))$, which is $\zeta$. 
 
 Suppose now that there is an other $t_2$ for which $\langle \nabla B_\zeta^p(a(t_2)), \dot{a}(t_2)\rangle=0$, then we also have $\nabla B_\zeta^p(a(t_2))=\hat{K}\dot{b}_{t_2}(0)$, for some $\hat{K}\in \mathbb{R}$, and we can assume as well that geodesic ray directed by $\dot{b}_{t_2}(0)$ has endpoint $\zeta$. Then, the geodesic triangle with vertices $\zeta$, $c(t_1)$ and $c(t_2)$ would have two right angles, and an angle equal to $0$, contradicting Gauss-Bonnet theorem: we have that the integral of the curvature of the surface on the interior region of the triangle, equals $\pi$ minus the sum of the interior angles of the triangle (see chapter 7 of \cite{ST} for a proof). As our surfaces has negative curvature, this integral should be strictly negative,   so the interior angles of the triangle cannot have sum equal to $\pi$. Then, the derivative of $B_\zeta^p(a(t))$ can only vanish for one value of $t$, as we wanted to see. 
 
 Now we are going to prove the second statement. As the Busemann function and $a(t)$ are continuous, $B_\zeta^p(a(t))$ is continuous as well. In one hand we have $B_\zeta^p(a(0))=0$, and in the other hand $\mathscr{C}_c(p)$ only meets at most in one other point the set of level $0$ of $B_\zeta^p$. Then, choosing the appropriate orientation for $a(t)$, we have that $|B_\zeta^p(a(t))|$ is an increasing function of $t$. As $|B_\zeta^p(a(t))|$ is a continuous function of $t$, for every $\epsilon>0$ there is $\delta(\epsilon)>0$ such that if $d(p,a(t))>\delta(\epsilon)$ then $|B_\zeta^p(a(t))|>\epsilon$, and as $|B_\zeta^p(a(t))|$ is increasing, then $\delta$ is also increasing.

 \end{proof}

  \begin{Proposition}\label{prop2}
 Consider $\zeta, \eta \in \partial_\infty \tilde{M}$ with $\zeta\neq \eta$, and a point $p\in \tilde{M}$. 
  Then we have:
 \begin{enumerate}
 \item The set $\mathscr{H}_\eta (q)\cap \mathscr{H}_\zeta (p)$ contains at most two points.
 \item If $\eta$ is the fixed point of a parabolic isometry $\gamma$ then, and $z\in \mathscr{H}_\eta (q)\cap \mathscr{H}_\zeta (p)$, then up to changing the orientation of $\gamma$, we have that for every $\epsilon>0$ there is a $\delta(\epsilon)$, with $\delta(\epsilon)$ being an increasing function of $\epsilon$, such that if $x\in [z,\gamma(z)]_{\mathscr{H}_\eta (q)}$ and $d(z,x)>\epsilon$, then $|B_\zeta(x,z)|>\delta(\epsilon)$. In particular $|B_\zeta(z,\gamma(z))|>\delta(\epsilon)$.
\end{enumerate}   
 \end{Proposition}
 
 \begin{proof}
 As in proposition \ref{prop1}, to show the first statement we are going to see that the derivative of $B_\zeta^p$ along the curve $\mathscr{H}_\eta (q)$ vanishes at most in one point, where $B_\zeta^p(z)=B_\zeta(p,z)$ for all $z\in \tilde{M}$. And then, $\mathscr{H}_\eta (q)$ can meet a level set of $B_\zeta^p$ in at most two points.
 
We first give an arc length parametrization $a(t)$ to the curve  $\mathscr{H}_\eta (q)$, such that $q=a(0)$.  The derivative of $B_\zeta^p(z)$ on the direction of $a(t)$ vanishes on a point $a(t_0)$ if and only if $\langle \nabla B_\zeta^p(a(t_0)),\dot{a}(t_0)\rangle=0$. Then $\nabla B_\zeta^p(a(t_0))$ is normal to the curve $\mathscr{H}_\eta (q)$, and then there is $k\in \mathbb{R}$ such that $\nabla B_\zeta^p(a(t_0))=k\nabla B_\eta^q(a(t_0))$, since the gradient of the Busemann function $B_\eta^q$ is perpendicular to its level set. 

Then, the geodesic directed by the vector $\nabla B_\zeta^p(a(t_0))$ has endpoints $\zeta$ and $\eta$, since the gradient of the Busemann function $B_\eta^q$ is parallel to the geodesic joining $q$ and $\eta$ (see proposition 3.2 of \cite{WB}). If there is an other point $a(t_1)$ such that $\langle \nabla B_\zeta^p(a(t_1)),\dot{a}(t_1)\rangle=0$, then we would also have that the geodesic directed by the vector $\nabla B_\zeta^p(a(t_1))$ is the geodesic joining $\eta$ and $\zeta$. As the geodesic joining two points is unique, this geodesic would be meeting $\mathscr{H}_\eta (q)$ in two points: $a(t_0)$ and $a(t_1)$. But a geodesic ending at $\eta$ only can meet a level set of $B_\eta^q$ once, as $B_\eta^q$ is increasing (or decreasing) along  geodesics having $\eta$ as one of its endpoints\footnote{This is also a consequence of Gauss-Bonnet theorem, as a triangle could not have an angle equal to $\pi$}.  Then, $a(t_0)=a(t_1)$, as we wanted.

The proof of the second statement is analogous to the second statement of proposition \ref{prop1}, using that, up to changing orientation of $\gamma$, the derivative of $B_\zeta^p$ has constant sign along the arc of the curve $\mathscr{H}_\eta (q)$ containing $z$ and $\gamma(z)$. 
 \end{proof}

\section{Proof of Theorem \ref{teorema}}
We are now going to prove Theorem \ref{teorema} in several steps. First, we remind the statement of the theorem:

\emph{
Let M be an orientable geometrically infinite surface with a complete Riemanninan metric of negative curvature. Let $v\in T^1M$ such that $v[0,\infty)$ is an almost minimizing geodesic ray with finite injectivity radius $a$, and such that $h_\mathbb{R}(v)$ is not closed. Then there is a sequence of times $\tau_n$ going to $\infty$ such that $g_{\tau_n}(v)\in \overline{h_\mathbb{R}(v)}$ for all $n$. Moreover, the set  $\mathcal{I}=\{t\in \mathbb{R}:g_t(v)\notin \overline{h_\mathbb{R}(v)}\}$ only contains intervals of bounded length.}

 Let $v\in T^1M$ be as in the hypothesis of \ref{teorema}, this is that $v[0,\infty)$ is an almost minimizing geodesic ray, with finite injectivity radius $a$, and such that $h_\mathbb{R}(v)$ is not a closed horocycle. Let $\tilde{v}$ be a lifted vector of $v$ on $T^1\tilde{M}$. We will call $\xi$ the point $\tilde{v}(\infty)\in \partial_\infty \tilde{M}$.

\begin{Lemma}\label{lema1}
There is a sequence $\{\tilde{v}_n\}_{n\in \mathbb{N}}\subset T^1\tilde{M}$ such that 

\begin{enumerate}
\item $\tilde{v}_n(0)=\gamma_n(\tilde{v}(0))$ for some $\gamma_n\in \Gamma$. 
\item  $\tilde{v}_n(\infty)=\tilde{v}(\infty)=\xi$.
\item $v_n \xrightarrow[n\rightarrow{}\infty]\, v$  en $T^1M$, where $v_n$ are the projected vectors of $\tilde{v}_n$ on $T^1M$. 
\item $B_\xi(\tilde{v}_n(0),\tilde{v}(0))\in [b,B]$ for some $0<b<B<\infty$, and for all $n\in \mathbb{N}$.
\end{enumerate}
\end{Lemma}

 \begin{proof}[Proof of lemma \ref{lema1}]
 From the definition of injectivity radius (\ref{injrad}), as $\underline{Inj}(v[0,\infty))=a$, we have a sequence $\{t_n\}$ going to $\infty$ such that $Inj (v(t_n))\xrightarrow[n\rightarrow{}\infty]\, a$. Then, given $\epsilon>0$, for big values of $n$ we have that $Inj (v(t_n))<a+\epsilon$. Then, if $\tilde{v}\in T^1\tilde{M}$ is a lifted vector of $v$, as $\tilde{v}(t_n)\in \tilde{M}$ is a lifted point of $v(t_n)$, there must be an isometry $\gamma_n \in \Gamma$ such that $d(\tilde{v}(t_n), \gamma_n(\tilde{v}(t_n))<a+\epsilon$. Then, the geodesic $\tilde{\alpha}_n$ joining  $\tilde{v}(t_n)$ and $\gamma_n(\tilde{v}(t_n))$, projects to a closed curve $\alpha_n$ in $M$.

 Now, we are going to construct the sequence $\{\tilde{v}_n\}$ as follows: consider the curve $\beta^T_n$ obtained by concatenation of $v[0,t_n]$, $\alpha_n$, and $v[t_n,T]$ in that order, where $T$ is any big number. This curve $\beta^T_n$ is not necessarily a geodesic. We call then $\hat{\beta}^T_n$ the geodesic joining $v(0)$ and $v(T)$ which is homotopic to $\beta^T_n$ relative to the endpoints. Now, as $T$ goes to $\infty$, $\hat{\beta}^T_n$ converges to a geodesic ray $v_n[0,\infty)$ starting at $v(0)$ which is asymptotic to $v[0,\infty)$ (see figures below). Then, it follows that $v_n\xrightarrow[n\rightarrow{}\infty]\, v$ in $T^1M$.
 
 \begin{figure}[ht!]  
\label{horo2}
\begin{center}   
\includegraphics[width=170mm]{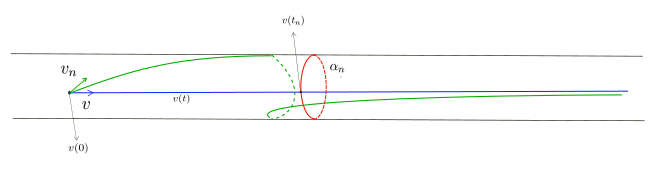}
 \end{center} 
 \caption{Geodesics $\alpha_n$, $v[0,\infty)$ and $v_n[0,\infty)$ on the surface $M$.}
 \end{figure}

\begin{figure}[ht!]  
 \label{horo1}
\begin{center}   
\includegraphics[width=100mm]{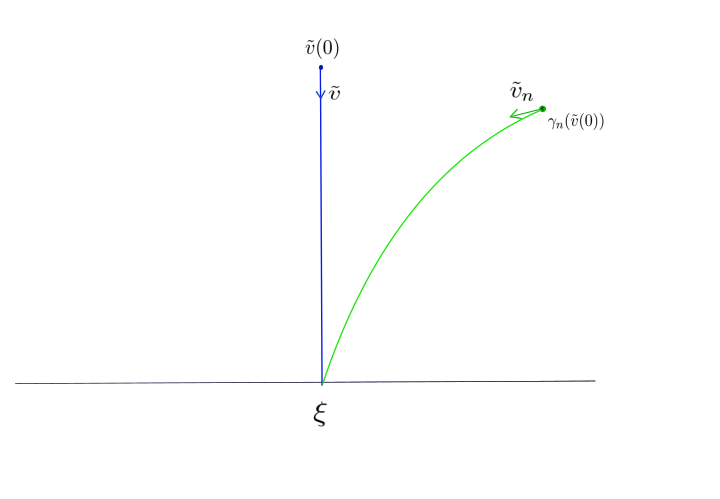}
 \end{center} 
 \caption{Geodesics $\tilde{v}[0,\infty)$ and $\tilde{v}_n[0,\infty)$ on the surface $\tilde{M}$.}
 \end{figure}
 
  \begin{figure}[ht!]  
\label{horo2}
\begin{center}   
\includegraphics[width=120mm]{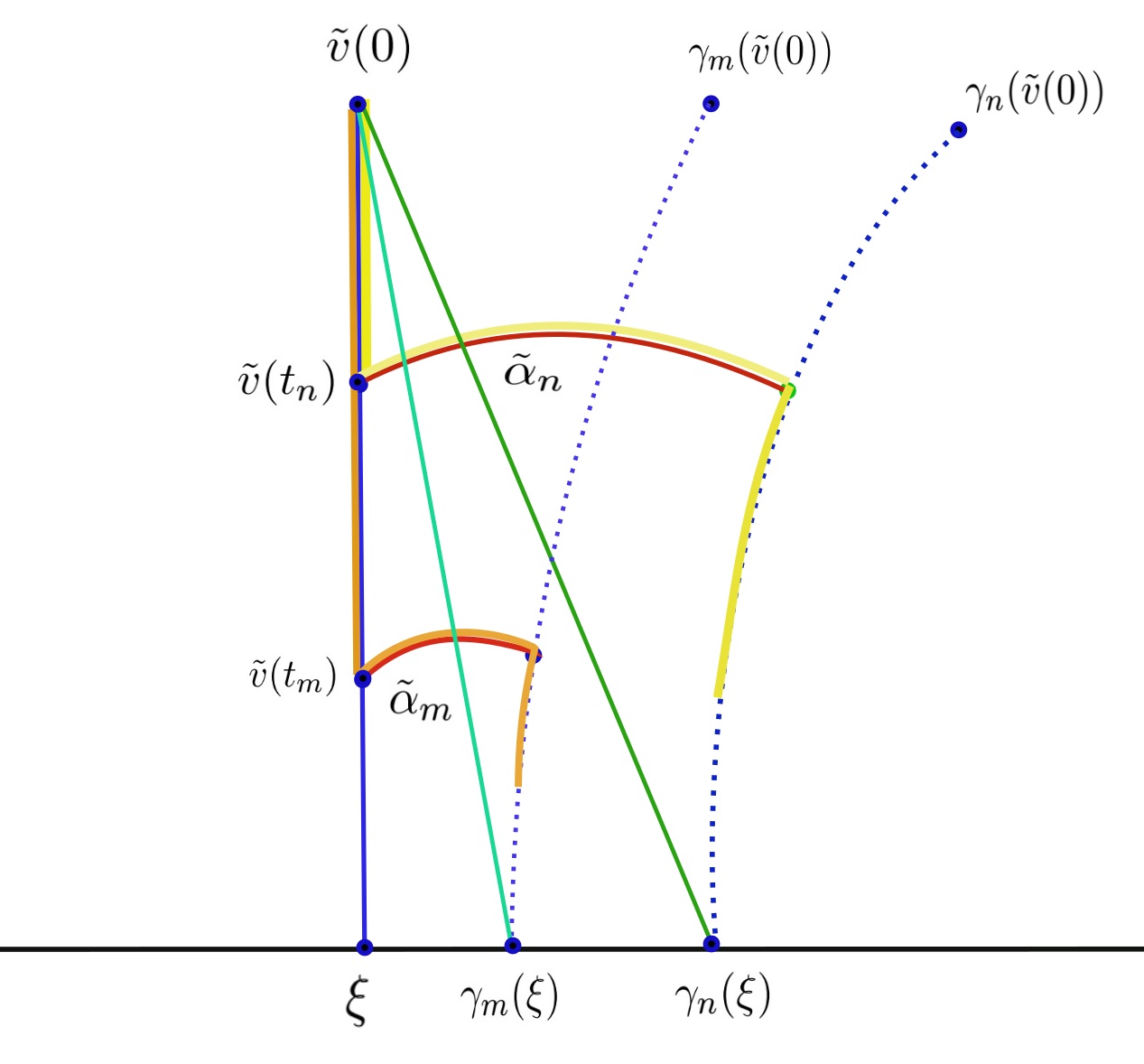}
 \end{center} 
 \caption{Here we see $\beta^T_n$ (in yellow) and $\beta^T_m$ (in orange), with $m>n$, both ending at time $T$. The dark green geodesic ray would be the lifted of $v_n[0,\infty))$, and the one in light green would be the lifted of $v_m[0,\infty)$.} 
 \end{figure}

  Now consider the lifted $\tilde{v}$ of $v$, which has basepoint $\tilde{v}(0)$, and let $\tilde{v}(\infty)=\xi$. Then there is a lifted $\tilde{v}_n$ of $v_n$ which also has endpoint $\xi$, because $v$ and $v_n$ are asymptotic. But $v_n[0,\infty)$ is the limit of $\hat{\beta}_n^T$, and it is homotopic to $\beta_n^T$. By construction of $\beta_n^T$, a lifted $\tilde{\beta}_n^T$ of this curve that starts at $\tilde{v}(0)$ must end at $\gamma_n(\tilde{v}(T))$. Then, we have that $\tilde{v}_n(0)$ must be $\gamma_n(\tilde{v}(0))$ (see figure 4). And then $\tilde{v}_n$ satisfies statements 1 and 2 of the lemma.

Let us see now that $B_\xi(\gamma_n(\tilde{v}(0)), \tilde{v}(0))\in [b,B]$.  Up to taking some positive power of $\gamma_n$, we can assume that the length of $\alpha_n$ is between $a+\epsilon$ and $2(a+\epsilon)$. Because of the construction of $v_n$, the length of $\alpha_n$ is the difference of lengths between $v[0,T]$ and $\beta_n^T$. As $\hat{\beta}_n^T$ is a geodesic with the same initial and endpoints as $\beta_n^T$ and in the same homotopy class, it is shorter than $\beta_n^T$. Then, the difference of length between $\hat{\beta}_n^T$ and $v[0,T]$ is bounded from above by the length of $\alpha_n$. As $T$ goes to $\infty$ this bound still holds, then the length of $\alpha_n$ is also an upper bound for $B_\xi(\gamma_n(\tilde{v}(0)), \tilde{v}(0))$.

Now, we are going to see that there is a lower bound for $B_\xi(\gamma_n(\tilde{v}(0)), \tilde{v}(0))$. There are then two cases:

Case 1: $\gamma_n$ is an hyperbolic isometry. Consider the arc joining $\tilde{v}(0)$ and $\gamma_n(\tilde{v}(0))$, contained on the curve $\mathscr{C}_{\gamma}(\tilde{v}(0))$, the curve of points at a constant distance from the axis of $\gamma_n$. Consider the function $\delta(\epsilon)$ given by Proposition \ref{prop1}, associated with the Busemann function $B_\xi^{\tilde{v}(0)}$, where $B_\xi^{\tilde{v}(0)}(z)=B_\xi(\tilde{v}(0),z)$ for all $z\in \tilde{M}$.

Now we have two possibilities: if $d(\gamma_n(\tilde{v}(0)), \tilde{v}(0))>\delta^{-1}(a+\epsilon)$, because of statement 2 of Proposition \ref{prop1}, we have that $B_\xi(\gamma_n(\tilde{v}(0)), \tilde{v}(0))>a+\epsilon$ for all $n$. And also $B_\xi(\gamma_n(\tilde{v}(0)), \tilde{v}(0))<length(\alpha_n)<2(a+\epsilon)$.

If $d(\gamma_n(\tilde{v}(0)), \tilde{v}(0))<\delta^{-1}(a+\epsilon)=C$, we can take a positive integer $k_n$ such that $d(\gamma_n^{k_n}(\tilde{v}(0)), \tilde{v}(0)) \leq k_n d(\tilde{v}(0), \gamma_n(\tilde{v}(0)))$. And choosing the right $k_n$ we can make that $d(\gamma_n^{k_n}(\tilde{v}(0)), \tilde{v}(0))\in [C, 2C]$, as we are taking an integral multiple of a number which is smaller than $C$, and also for every $n$, we have that $d(\tilde{v}(0),\gamma_n^d(\tilde{v}(0)))\xrightarrow[d\rightarrow{}\infty]\, \infty$. Then $B_\xi(\gamma_n^{k_n}(\tilde{v}(0)), \tilde{v}(0))\in [\delta(C),\delta(2C)].$ 
Finally substituting $\gamma_n$ by $\gamma_n^{k_n}$, we have what we wanted.

Case 2: $\gamma_n$ is a parabolic isometry. Suppose $\gamma_n$ has fixed point $\eta$. Consider the arc joining $\tilde{v}(0)$ and $\gamma_n(\tilde{v}(0))$. Both $\tilde{v}(0)$ and $\gamma_n(\tilde{v}(0))$ are in the same projected horocycle based at $\eta$, meaning that $B_\eta (\gamma_n(\tilde{v}(0)), \tilde{v}(0))=0$. Now, with an analogous reasoning to the hyperbolic case,we can choose $k_n$ such that $d(\gamma_n^{k_n}(\tilde{v}(0)), \tilde{v}(0))\in [C,2C]$, where $C=\delta^{-1}(a+\epsilon)$ and $\delta$ is as in the second statement of Proposition \ref{prop2}. Substituting $\gamma_n$ by $\gamma_n^{k_n}$ we get $$B_\xi(\gamma_n(\tilde{v}(0)), \tilde{v}(0))\in [b,B],$$ for all $n\in \mathbb{N}$, as we wanted to see.

 \end{proof}
 
\begin{Lemma}\label{lema2}
For a sequence $\{\tilde{v}_n\}_{n\in \mathbb{N}}\subset T^1\tilde{M}$ satisfying points 1 to 4 of Lemma \ref{lema1}, there is a sequence $\{r_n\}_{n\in \mathbb{N}}\subset \mathbb{R}$, with $r_n \xrightarrow[n\rightarrow{}\infty]\, r_0\in \mathbb{R}$ such that $d(g_{t+r_n}(v_n),g_t(v))\xrightarrow[t\rightarrow{}\infty]\, 0$.
\end{Lemma}

 \begin{proof}[Proof of lemma \ref{lema2}] 
 
As $b<B_\xi(\tilde{v}_n(0),\tilde{v}(0))<B$, if we define $r_n:=B_\xi(\tilde{v}_n(0),\tilde{v}(0))$, as it is a bounded sequence, replacing if it is necessary $r_n$ by a subsequence, we can assume $r_n\longrightarrow r_0\in [b,B]$. By definition of the Busemann function, we also have: $d(g_t(v),g_{t+r_n}(v_n)) \xrightarrow[t\rightarrow{}\infty]\, 0$. 
 
  \end{proof}
  
  \begin{proof}[Proof of theorem \ref{teorema}]
  It follows as a direct corollary of Lemma \ref{lema1}, Lemma \ref{lema2} and proposition \ref{Prop}, that there is $r_0$, with $0<r_0<\infty$ such that  $g_{r_0}(v)\in \overline{h_\mathbb{R}(v)}$. Now $g_{r_0}(g_{r_0}(v))\in g_{r_0}(\overline{h_\mathbb{R}(v)})=\overline{h_\mathbb{R}(g_{r_0}(v))}\subset \overline{h_\mathbb{R}(v)}$, because the closure of an orbit is invariant by the horocycle flow. Applying the same argument, defining $\tau_n:=nr_0$, we can conclude that $g_{\tau_n}(v)\in \overline{h_\mathbb{R}(v)}$ for all $n$. As $r_0\in [b,B]$, in every set of length $B$ there is at least one of these $\tau_n$. This completes the proof. 
  \end{proof}
  
  \begin{Remark}
  From the proof of Lemma \ref{lema1} we have that there is $t_0\in [b,B]$ such that $g_{t_0}(v)\in \overline{h_\mathbb{R}(v)}$, were $B=\delta(2\delta^{-1}(a+\epsilon))$. As $\delta$ measures the variation of the Busemann function along a curve which is not a geodesic (along which the biggest variation occurs), we have that $B<2(a+\epsilon)$. Then, even when $\delta$ depends on $\gamma$, it is uniformly bounded from above. Also, as we can take $\epsilon$ as small as we want, in every interval of length $2a$ there is at least one $t$ such that $g_{t}(v)\in \overline{h_\mathbb{R}(v)}$.
  \end{Remark}

  \begin{Corollary}
  In case $a=0$, one has $g_{\mathbb{R}^+}(v)\subset \overline{h_\mathbb{R}(v)}$.
  \end{Corollary}
  
  \begin{proof}
  It suffices to observe that the functions $\delta$ in both propositions \ref{prop1} and \ref{prop2} satisfy $\delta(0)=0$, and apply Theorem \ref{teorema}.
  \end{proof}

\section{Applications to tight surfaces}

\begin{Remark}\label{rem3}
The fundamental groups of tight surfaces are infinitely generated and of the first kind. The last means that, if $\tilde{M}$ is the universal cover of $M$, then $L(\Gamma)=\partial_\infty \tilde{M}$.
\end{Remark}

\begin{Remark}\label{rem4}
In view of proposition \ref{propdense} and remark \ref{rem3}, for a tight surface $M$, if $\tilde{v}\in T^1\tilde{M}$ is such that $\tilde{v}(\infty)$ is an horocyclic limit point, then $\overline{\hat{\pi}(h_\mathbb{R}(\tilde{v}))}=T^1M$. 
\end{Remark}

\begin{Proposition}\label{injtight}
Every almost minimazing geodesic ray $v[0, \infty)$ on a tight surfaces has finite injectivity radius. 
 \end{Proposition}
 
 \begin{proof}
 Consider the geodesics $\delta_n$ of definition \ref{tight}. As $v[0, \infty)$ is almost minimizing, it intersects an infinite number of these geodesics. Replacing $\delta_n$ by a subsequence, we can assume that $v[0, \infty)$ intersects $\delta_n$ for all $n$. Let $t_n$ be the times such that $v(t_n)\in \delta_n$. Let $\tilde{\delta_n}$ be a lift of $\delta_n$ on $\tilde{M}$, and $\tilde{v}(t_n)$ a lift of $v(t_n)$. Consider $\eta_n \in \Gamma$ the hyperbolic isometry fixing $\tilde{\delta}_n$. By hypothesis, we have that the lengths of $\delta_n$ are bounded by a constant $A$, so $d(\tilde{v}(t_n), \eta_n(\tilde{v}(t_n)))\leq A$. This means that $Inj(v(t_n))\leq A$ for all $n$, and then $\liminf_{n\longrightarrow \infty} Inj(v(t_n))\leq A$. It follows that $\liminf_{t\longrightarrow \infty} Inj(v(t))<A$. Then we have $$\underline{Inj}(v[0,\infty))\leq A<\infty,$$ as we wanted.
 \end{proof}

\begin{Definition}\em
Given a metric space $Y$ and a flow $\{\varphi_t\}_{t\in \mathbb{R}}$, we say that $X\subset Y$ is a \textit{minimal set} for the flow, if it is closed, invariant by $\varphi_t$ and minimal with respect to the inclusion.
\end{Definition}

\begin{proof}[Proof of Corollary \ref{theorema2}]
Supose $X\subset T^1M$ is a minimal set for the horocycle flow. Consider a vector $v\in X$ and $\tilde{v}\in T^1\tilde{M}$ a lifted of $v$. As $X$ is a minimal set, $h_\mathbb{R}(v)$ must be dense in $X$, otherwise its closure would be a proper invariant subset of $X$, and $X$ would not be minimal. In the other hand, $X$ can not be $T^1M$, because in that case every orbit should be dense, but that can't happen since the limit set has both horocyclic and nonhorocyclic limit points, and horocycles based on nonhorocyclic limit points are not dense. So $X$ is a proper subset of $T^1M$. This implies that $\tilde{v}(\infty)$ must be a nonhorocyclic limit point, since the closure of its projected orbit is $X\neq T^1M$. Then, $v[0,\infty)$ is an almost minimizing geodesic ray, and by theorem \ref{teorema} and proposition \ref{injtight}, we know that there is a $t_0$ such that $g_{t_0}(v)\in \overline{h_\mathbb{R}(v)}$. Then, $g_{t_0}(X)\cap X\neq \emptyset$ and then $g_{t_0}(X)=X$. Then for all $n\in \mathbb{N}$ we have $g_{nt_0}(X)=X$.

Let us see that it actually implies that $\tilde{v}(\infty)$ is horocyclic: consider an horocycle $B$ based at  $\tilde{v}(\infty)$. As we know, we can write $$B=\{z\in \tilde{M}:\ B_{\tilde{v}(\infty)}(\tilde{v}(0),z)>k\},$$ for some $k\in \mathbb{R}$. As $v(nt_0)\in X=\overline{h_\mathbb{R}(v)}$, and because of proposition \ref{Prop}, there is a sequence $\{\gamma_m\}_{m\in \mathbb{N}}\subset \Gamma$ such that $B_{\tilde{v}(\infty)}(\tilde{v}(0),\gamma_m^{-1}\tilde{v}(0))\xrightarrow[m\rightarrow{}\infty]\, nt_0$. Then, choosing a big $n$ we have $nt_0>k$, and then for a big $m$, $\gamma_m^{-1}\tilde{v}(0) \in B$. So we can find an element of the $\Gamma$-orbit of $\tilde{v}(0)$ on any horocycle based at $\tilde{v}(\infty)$, and this means that $\tilde{v}(\infty)$ is an horocyclic limit point, which is absurd as we already showed that it must be nonhorocyclic. 
\end{proof}

 \end{document}